\documentclass{elsarticle}

\usepackage{amsfonts, amsmath, amssymb,latexsym}
\usepackage{array}
\usepackage{bigstrut}

\newtheorem{lemma}{Lemma}
\newtheorem{theorem}{Theorem}
\newtheorem{corollary}{Corollary}

\newproof{proof}{Proof}

\usepackage{graphicx}   
\usepackage{verbatim}   
\usepackage{color}      
\usepackage{subfigure}  
\usepackage{hyperref}   
\usepackage{fullpage}
\usepackage{tikz}
\usepackage{amsfonts}

\makeatletter
\newcommand\ackname{Acknowledgements}
\if@titlepage
  \newenvironment{acknowledgements}{%
      \titlepage
      \null\vfil
      \@beginparpenalty\@lowpenalty
      \begin{center}%
        \bfseries \ackname
        \@endparpenalty\@M
      \end{center}}%
     {\par\vfil\null\endtitlepage}
\else
  \newenvironment{acknowledgements}{%
      \if@twocolumn
        \section*{\abstractname}%
      \else
        \small
        \begin{center}%
          {\bfseries \ackname\vspace{-.5em}\vspace{\z@}}%
        \end{center}%
        \quotation
      \fi}
      {\if@twocolumn\else\endquotation\fi}
\fi
\makeatother

\newcommand{\xs}{|X_{q+1}|}
\newcommand{\ws}{|X_{q+2}|}
\newcommand{\ys}{|X_{q}|}

\newcommand{\xss}{X_{q+1}}
\newcommand{\yss}{X_{q}}
\newcommand{\zss}{X_{\leq q-1}}

\begin{document}

\title{Extremal Graphs Without 4-Cycles}

\author[FF]{
Frank A. Firke}
\author[PK]{Peter M. Kosek}
\author[EN]{Evan D. Nash}
\author[JW]{Jason Williford}

\address[FF]{Department of Mathematics, Carleton College, Northfield, Minnesota 55057,  USA}

\address[PK]{Department of Mathematics, The College at Brockport, State University of New York, Brockport, NY 14420, USA}

\address[EN]{Department of Mathematics, University of Nebraska-Lincoln,  Lincoln, NE 68588, USA}

\address[JW]{ Department of Mathematics, University of Wyoming, Laramie, Wyoming 82071, USA}

\begin{abstract}
We prove an upper bound for the number of edges a $C_4$-free graph on $q^2+q$ vertices can contain for $q$ even. This upper bound is achieved whenever there is an orthogonal polarity graph of a plane of even order $q$.
\end{abstract}

\maketitle

Let $n$ be a positive integer and $G$ a graph. We define $ex(n,G)$ to be the largest number of edges possible in a graph on $n$ vertices that does not contain $G$ as a subgraph; we call a graph on $n$ vertices \emph{extremal} if it has $ex(n,G)$ edges and does not contain $G$ as a subgraph. $EX(n,G)$ is the set of all extremal $G$-free graphs on $n$ vertices.

The problem of determining $ex(n,G)$ (and $EX(n,G)$) for general $n$ and $G$ belongs to an area of graph theory called \emph{extremal graph theory}. Extremal graph theory officially began with Tur\'an's theorem that solves $EX(n,K_m)$ for all $n$ and $m$, a result that is striking in its precision. In general, however, exact results for $ex(n,G)$ (and especially $EX(n,G)$) are very rare; most results are upper or lower bounds and asymptotic results.  For many bipartite $G$ there is a large gap between upper and lower bounds.

The question of $ex(n,C_4)$ (where $C_4$ is a cycle of length 4) has an interesting history; Erd\H{o}s originally posed the problem in 1938, and the bipartite version of this problem was solved by Reiman using a construction derived from the projective plane (see \cite{B} and the references therein for a more detailed history). Reiman also determined the upper bound $ex(n,C_4) \leq \frac n 4 (1 + \sqrt{4n-3})$ for general graphs, but this is known not to be sharp \cite{R}. Erd\H{o}s, R\'enyi, and S\'os later showed that this is asymptotically correct using a construction known as the \emph{Erd\H{os}-R\'enyi graph} derived from the orthogonal polarity graph of the classical projective plane \cite{ER} \cite{ERS}.  This is part of a more general family of graphs which we define below.  

Let $\pi$ be a finite projective plane with point set $P$ and line set $L$.  A polarity $\phi$ of $\pi$ is an involutionary permutation of $P \cup L$ which maps points to lines and lines to points and reverses containment.  We call points absolute when they are contained in their own polar image.   A polarity is called orthogonal if there are exactly $q+1$ absolute points. We define the polarity graph of $\pi$ to be the graph with vertex set $P$, with two distinct vertices $x,y$ adjacent whenever $x \in \phi(y)$.  The graph is called an orthogonal polarity graph if the polarity is orthogonal.  This graph is $C_4$-free, has $q^2$ vertices of degree $q+1$, and $q+1$ of degree $q$, for a total of $\frac 1 2 q(q+1)^2$ edges.


F\"uredi determined the first exact result that encompasses infinitely many $n$, namely that for $q>13$ we have $ex(q^2+q+1,C_4) \leq \frac 1 2 q(q+1)^2$ \cite{F1} \cite{F2}, with equality if and only if the graph is an orthogonal polarity graph of a plane of order $q$.
 In particular, this shows $ex(q^2+q+1,C_4) = \frac 1 2 q(q+1)^2$ for all prime powers $q$.  

The question of finding $ex(n,C_4)$ exactly for general $n$ appears to be a difficult problem.  Computer searches by Clapham et al. \cite{CFS} and Yuansheng and Rowlinson \cite{YR} determined $EX(n,C_4)$ for all $n \leq 31$.  
More general lower bounds are given in \cite{A} by deleting carefully chosen vertices from the Erd\H{o}s-R\'enyi graph.  It is not known if any of these bounds are sharp in general.   In particular it is not even known whether deleting a single vertex of 
degree $q$ from an orthogonal polarity graph graph yields a graph which is still extremal, a question posed by Lazebnik in 2003 \cite{L1}.
More generally, is $ex(q^2+q,C_4) \leq \frac{1}{2}q(q+1)^2-q$?  In this paper we will prove the following theorem:

\begin{theorem}\label{main}
For $q$ even, $ex(q^2+q,C_4) \leq \frac 1 2 q(q+1)^2-q$.  
\end{theorem}

It follows that equality holds for all $q$ which are powers of 2.

The question of determining $EX(q^2+q,C_4)$ in this case is subtler; the searches referred to above showed that there are multiple constructions that achieve the bound for $q = 2, 3$, but for $q = 4,5$ there is only one. In a subsequent paper, we will prove the following: 

\begin{theorem} For all but finitely many even $q$, any $C_4$-free graph with $ex(q^2+q,C_4)$ edges is derived from an orthogonal polarity graph by removing a vertex of minimum degree. \end{theorem}

The proof of this result is much more lengthy and complicated than that of the inequality in Theorem \ref{main}, and requires $q$ to be sufficiently large.  The purpose of this paper is to give a simpler proof of the inequality and show it holds for all even $q$.  We start with some notation.

We let $X_k$ be the set of vertices of degree $k$, $X_{\leq k}$ be the set of vertices of degree at most $k$, $E_0$ be $\frac 1 2 q(q+1)^2-q$, and $n$ be the number of vertices ($q^2+q$). We will use $\Gamma ( x )$ to represent the vertices in the neighborhood of $x$. For our various lemmas, we will specify in each case whether $q$ is an even number or simply a positive integer; however, in all cases we consider $q \geq 6$. (We know from \cite{CFS} and \cite{YR} that the inequality in Theorem \ref{main} is true for $q \leq 5$.)

In general, we proceed indirectly. We will show that no $C_4$-free graph with $E_0+1$ edges can exist, from which we conclude that a graph cannot have more than that number of edges (as it would contain an impossible subgraph).
We will use and generalize the techniques found in \cite{F1}, \cite{F2}, and \cite{F3}.  

\begin{lemma}
Let $q$ be a natural number greater than 2 and let $G$ be a $C_4$-free graph on $q^2+q$ 
vertices with at least $E_0$ edges.  Then 
the 
maximum degree 
of a vertex 
in $G$ is at most $q+2$.
\end{lemma}

\begin{proof}

Let $u$ be a vertex of 
$G$ of maximum degree 
$d$.  Let $e$ be the number of edges 
of $G$, $e \geq \frac{1}{2}q(q+1)^2-q$.  We proceed by bounding the 
number of 2-paths in $G$ which have no endpoints in $\Gamma ( v )$.
This gives us:

$$\binom{n-d}{2} \geq \sum\limits_{v \neq u} \binom{d(v)-1}{2} $$

Using Jensen's inequality for the function $f(x) = \binom{x}{2}$ we have:

$$\sum\limits_{ v \neq u } \binom{d(v)-1}{2} 
 \geq (n-1)\binom{(2e-(n-1)-d)/(n-1)}{2}$$

Multiplying by $2 (n-1) (q+1)$ and simplifying yields:
\begin{equation}\label{inq}
(q+1)(n-1)(n-d)(n-d-1) \geq (q+1)(2e-n-d+1)(2e-2n-d+2)
\end{equation}

However, we also have:

$$
(q+1)(2e-2n-d+2)-(n-1)(n-d-1) \geq (q^2-2)d-q^3-2q^2+q+1 
$$ $$\geq (q^2-2)(q+3)-q^3-3q^2+1 \geq q^2-q-5$$

with $q^2-q-5>0$ for $q>2$, \\

which gives us:

\begin{equation}\label{inq1}
(q+1)(2e-2n-d+2) > (n-1)(n-d). 
\end{equation}

We also have the inequality:

$$
(2e-n-d+1) - (q+1)(n-d) \geq -q^2-3q+1+qd 
$$ $$\geq -q^2-3q+1+q(q+3) =1>0,$$

which demonstrates that:

\begin{equation}\label{inq2}
(2e-n-d+1) > (q+1)(n-d-1). 
\end{equation}

Therefore the product of (\ref{inq1}) and (\ref{inq2}) contradict (\ref{inq}), 
and 
the theorem 
follows.
\end{proof}

Since we now have an upper bound on the maximum degree of an extremal graph, we 
focus on 
the lower bound.

\begin{theorem} \label{degseqs}
Let $q$ be an even number and let $G$ be a $C_4$-free graph on $q^2+q$ vertices, 
with maximum degree $q+1$ 
or less.  
Then, if $e$ denotes the number of edges of $G$, we have $e \leq E_0$.  Furthermore, if equality holds, then the degree sequence of $G$ must be one of the following (where $z$ is a parameter):

\begin{center}
    \begin{tabular}{ |l | l | l | l |}
    \hline
    $|X_{q+1}|$ & $|X_q|$ &$|X_{q-1}|$ &$|X_{q-2}|$ \\ \hline
    $q^2-q+z$ & $2q-2z$ & $z$ & 0 \\ \hline 
    $q^2-q+z+1$ & $2q-2z-1$ & $z-1$ & 1 \\ \hline 
    \end{tabular}
\end{center}
\end{theorem}

\begin{proof}

If the maximum degree of $G$ is $q$ or less, we have $e \leq 
\frac{1}{2}q(q^2+q) \leq E_0$ and the 
theorem is immediate.  
  Therefore, we take $G$ to have maximum degree $q+1$.
  It is clear that:
$$\xs+\ys+|\zss|=q^2+q
$$

Noting that the degree sum of all the vertices of a graph is equal to $2e$, we 
have:

\begin{equation}\label{edge}
2e \leq (q+1)\xs+q\ys+(q-1)|\zss|=q(q^2+q) +\xs-|\zss|
\end{equation}

Let $v$ be a vertex of degree $q+1$.  We wish to bound the average degree of 
$\Gamma(v)$.
We will do this by noting that the collection $C(v)$ of vertices distance 2 
or less from 
$v$ must naturally be less than $q^2+q$.  We then show that $|\bigcup_{w \in 
\Gamma(v)} 
\Gamma(w)|+1 \leq |C(v)|$.

Each vertex in $\Gamma(v)$ is connected to at most one other vertex in 
$\Gamma(v)$, as 
otherwise it would imply a $C_4$ is in $G$.  Also, since $q+1$ is odd, there 
must be a 
vertex in $\Gamma(v)$ which is connected to no other vertex in $\Gamma(v)$.  
Rephrasing, 
this means that there is at least one vertex in $C(v)$ that is not in 
$\bigcup_{w \in 
\Gamma(v)} \Gamma(w)$.  We also know that for $w,u \in \Gamma(v)$ we have 
$\Gamma(w) 
\cup 
\Gamma(u) = \{v\}$, otherwise it again implies $G$ has a $C_4$.  We then have: 

$$
\left | \bigcup_{w \in \Gamma(v)} \Gamma(w) \right |+1= \left (\sum_{w \in \Gamma(v)} 
d(w)\right)-q+1 
\leq |C(v)| \leq q^2+q
$$

Then we have:

$$\frac{  \sum_{w \in \Gamma(v)} d(w)}{q+1} \leq 
\frac{q^2+2q-1}{q+1}=q+1-\frac{2}{q+1}
$$
Therefore, we can conclude that if a vertex of degree $q+1$ is connected to no 
vertex of degree $q-1$, 
then it must be connected to at least two vertices of degree $q$.  Let $A$ be the 
set of 
vertices of degree $q+1$ which are connected to at least two vertices of $\yss$ but no 
vertex of 
$X_{\leq q-1}$, 
and $B$ be the set of vertices of $\xss$ connected to at least one vertex of 
$X_{\leq q-1}$.
Let $a=|A|$ and  $b=|B|$.  Naturally $a+b=\xss$.

We consider the number of edges $e'$ with one endpoint in $A$ and the 
other in $\yss$.  
As 
each vertex of $A$ is connected to at least two vertices of $\yss$, and each 
vertex of 
$\yss$ 
is connected to at most $q$ vertices of $\xss$, we have:
  
\begin{equation}\label{ineq1}
2a \leq e' \leq q\ys
\end{equation}

Let $e''$ be the number of edges with one endpoint in $B$ and the other in 
$\zss$.
As each vertex of $B$ is connected to at least one vertex of $\zss$ and each 
vertex in 
$\zss$ 
is connected to at most $q-1$ vertices in $B$, we have:

\begin{equation}\label{ineq2}
b \leq e'' \leq (q-1)|\zss|
\end{equation}

Adding twice (\ref{ineq2}) to (\ref{ineq1}) we get:

$$2\xs \leq q\ys+(2q-2)|\zss|
$$
Adding $q\xs$ to both sides we have:

$$(q+2)\xs \leq q(\xs+\ys+|\zss|)+(q-2)|\zss|=q^3+q^2+(q-2)|\zss|
$$

Dividing both sides by (q+2) and expanding, we obtain:

$$\xs \leq q^2-q+|\zss|-\frac{4|\zss|}{q+2}-\frac{4}{q+2}+2
$$

This implies

 $$\xs-|\zss| \leq q^2-q + 1$$

Using this with (\ref{edge}) we have:

$$2e \leq q^3+q^2 +q^2-q +1
$$
Since $q$ and $2e$ are even, we must have:

$$e \leq \frac{1}{2}q(q+1)^2-q$$

If equality holds, then we have $ q^2-q \leq \xs-|\zss| \leq q^2-q+1 $ and $2e \leq (q+1)\xs + q\ys + (q-1)|\zss| \leq 2e+1$.  This implies that 
at most one of the vertices in $\zss$ has degree $q-2$, and the rest have degree $q-1$. The remainder of the theorem follows. \end{proof}

It is now clear that any graph with more than $E_0$ edges must have maximum degree equal to $q+2$.  The rest of the paper is devoted to showing that no such graph exists.  This is done by utilizing a connection between vertices of degree $q+2$ and vertices of relatively small degree.

\begin{lemma} \label{doubleconnect}
If there is a $C_4$-free graph on $n$ vertices and $E_0+1$ edges (with $q \in \mathbb{N}$), then any vertex of degree $\delta \leq \frac q 2 + 1$ connects to every vertex of degree $q+2$.
\end{lemma}

\begin{proof} Assume for the sake of contradiction that this statement is not true. Then there is a vertex $v$ of degree $\delta$ (where $\delta \leq \frac q 2 +1$) and a vertex $u$ of degree $q+2$ such that $u \nsim v$. Then we remove all $\delta$ edges in which $v$ is incident and add a new edge from $u$ to $v$, so that our graph now has $e = E_0+1-\delta$ edges, a vertex $v$ of degree 1, and a vertex $u$ of degree $q+3$.

Now we extend the lemma used in \cite{F1} concerning 2-paths containing no endpoints in $\Gamma(u)$, the neighborhood of $u$. We know that we can bound the number of such 2-paths above with ${{n-d(u)} \choose 2}$, as there can be at most one 2-path between any pair of points not in $\Gamma(u)$. We also know that each vertex that is not $u$ has at most one neighbor in $\Gamma(u)$, which means that this inequality must hold: $${{n -d(u)} \choose 2} \geq \sum_{x \neq u} {{d(x) - 1} \choose 2}$$ as the right side of that inequality is a lower bound on the number of actual 2-paths in the graph. 

To actually calculate this, we first note that every 2-path involving $v$ must have an endpoint in $\Gamma(u)$, which means that we can write $$\sum_{x \neq u} {{d(x) - 1} \choose 2} = \sum_{x \neq u,v} {{d(x)-1} \choose 2}$$ If we consider the total sum being chosen from (i.e. $\sum_{x \neq u,v} d(x) -1$) we get the number $$2e - (n-2)-(q+3)-1$$ (since the total degree sum is $2e$ and we subtract first the degrees of the two uncounted vertices and then $1$ from the remaining $n-2$ terms). We can thus use Jensen's inequality to obtain this expression: $${n-(q+3) \choose 2} \geq (n-2) {\frac{2e-(n-2)-(q+3)-1}{n-2} \choose 2}$$

Now we take $$e = E_0 - \frac q 2 + 1$$ (which corresponds to $\delta = \frac q 2 + 1$). Since the left side is not dependent on $e$ and the right side is, if the inequality fails for our chosen $e$ then it will certainly fail for larger values $e$, which is equivalent to smaller values of $\delta$. When we expand and simplify the above inequality, we find it is equivalent to the following: $$- \frac{2q^3-2q^2-10q+12}{(q^2+q-2)} \geq 0 $$ which is not true for any relevant $q$. \end{proof}

\begin{corollary} \label{dcc}
For any $q$, if $|X_{q+2}| \geq 2$ in a $C_4$-free graph with $E_0 + 1$ edges on $n$ vertices, there can be only one vertex $v$ of degree $\frac q 2 + 1$ or less. In that case, $|X_{q+2}| \leq d(v).$
\end{corollary}

\begin{proof} The first part follows from the prior lemma and the fact that the graph is $C_4$-free; the second part follows from the lemma and the first part of the corollary. \end{proof}

\begin{lemma} \label{total2plemma}
The maximum number of 2-paths in a graph with $n$ vertices and $E_0+1$ edges (with $q$ even) is $qe - \xs + \frac 1 2 \ws$.
\end{lemma}

\begin{proof}
Since there can only be one 2-path between any two vertices (if there are more, there would be a $C_4$ in the graph), we can bound the number of 2-paths by ${n \choose 2}$. However, this may be improved by bounding the number of pairs of vertices which are not the endpoints of a 2-path. We consider how many other vertices cannot be reached in two steps from a given vertex, a function we will denote by $f(v)$.  The exact number of 2-paths in the graph is ${n \choose 2} - \frac{1}{2} \sum\limits_v f(v)$, so any lower bound on $\sum\limits_v f(v)$ will in turn yield an upper bound on the number of 2-paths in the graph.

To bound $\sum\limits_v f(v)$, we can compute what $f(v)$ would be if $v$ is connected only to vertices of degree $q+1$, giving us a function we call $g(v)$.  This leads us to this table:


\begin{center}
    \begin{tabular}{ | l | l |}
    \hline
    $d(v)$ & $g(v)$\\ \hline
    $q-2$ & $3q-1$\\ \hline 
    $q-1$ & $2q-1$ \\ \hline
    $q$ & $q-1$\\ \hline
    $q+1$ & $1$\\ \hline
    $q+2$ & $0$\\ \hline
    \end{tabular}
\end{center}

As a sample calculation, if $d(v) = q$ then $g(v) = q^2 + q - (1 + q \cdot q) = q-1$ because it has $q$ neighbors that each have $q$ neighbors other than $v$. The $1$ corresponding to degree $q+1$ comes from the fact that a vertex of degree $q+1$ must have one neighbor it cannot be connected to in a 2-path since (by assumption) $q+1$ is odd.\\

Strictly speaking, the values of $g(v)$ are not lower bounds on $f(v)$, since there are also vertices of degree $q+2$.  In general, if $v$ is adjacent to $k$ vertices of degree $q+2$, we have $f(v) \geq g(v)-k$.  Then subtracting $(q+2)|X_{q+2}|$ from $\sum\limits_v g(v)$ gives us the bound  
$\sum\limits_v f(v) \geq \sum\limits_v g(v) - (q+2)|X_{q+2}|$.
We note that, for $d(v) \leq q$, $g(v) = q( q+1 - d(v)) -1$; for $d(v) = q+1$, we must add 2 to that formula, while for $d(v) = q+2$ we must add $q+1$. 

This allows us to establish an upper bound on the number of 2-paths in the graph as follows: \begin{align*} {n \choose 2} - \frac 1 2 \sum_{v \in V(G)} f(v) & \leq {q^2 + q \choose 2} - \frac 1 2 \sum_{v \in V(G)} (g(v) - |X_{q+2}|) \\ &= \frac 1 2 [ (q^2+q)(q^2+q-1) - (\sum_{v \in V(G)} (q( q+1 - d(v)) -1 ) + 2\xs - \ws)] \\ &= \frac 1 2 [(q^2+q)(q^2+q-1) - |V(G)|(q^2+q-1) -2\xs + \ws + q\sum_{v \in V(G)} d(v) ]\\ &= qe - \xs + \frac 1 2 \ws  \end{align*} and so we have our result. \end{proof}

\begin{lemma}
For any $C_4$-free graph $G$ on $n$ vertices with $E_0+1$ edges ($q$ even), $\delta(G) > \frac q 2 +1$.
\end{lemma}

\begin{proof}
Assume for the sake of contradiction that there is a vertex $v$ such that $d(v) = \delta \leq \frac q 2 +1$. We know from Lemma \ref{doubleconnect} and Corollary \ref{dcc} that $v$ is unique and $v$ is connected to every vertex in $X_{q+2}$. We also know from the Lemma \ref{total2plemma} that this inequality must hold: \begin{equation} \label{2path} qe - \xs + \frac 1 2 \ws \geq \ws {{q+2} \choose 2} + \sum_{x \neq v, \not \in X_{q+2}} {{d(x)} \choose 2} + {{\delta} \choose 2} \end{equation} since the right hand side is the total number of 2-paths in the graph.

Take $A$ to be the set of vertices that are neither $v$ nor in $X_{q+2}$. We wish to find the average degree of $A$, which we will denote by $c$. Since the maximum degree of any vertex in $A$ is $q+1$, $c \leq q+1$. Moreover, $c$ will be minimized when the $\ws = \delta = \frac q 2 +1$; in that case, we can calculate $c$: $$ c = \frac{2(E_0+1) - (\frac q 2 +1) - (\frac q 2 +1)(q+2)} {n - (\frac q 2 + 1 + 1)} $$ which yields $c = q+1 - \frac{4q-2}{2q^2+q-4}$. Since the subtracted term is less than 1 for all $q$, $q < c \leq q+1$. 

Now, clearly, for a fixed $\ws$ the left side of \ref{2path} is maximized when $\xs$ is minimized. We also wish to minimize the term $M = \sum_{x \in A} {{d(x)} \choose 2}$. If there is a vertex $y$ of degree $q - k$ (for some integer $k \geq 1$), then we know that we keep the same degree sum in $A$ (which is fixed, since $\delta$ and $\ws$ are fixed) if we were to take a vertex of degree $q+1$, turn it into a vertex of degree $q$, and raise the degree of $y$ by 1. Moreover, this will actually decrease $M$, because of the following arithmetic: \begin{align*} {q \choose 2} - {{q+1} \choose 2} + {{q - k +1} \choose 2} - {{q-k} \choose 2} &= -2q + 2(q-k) \\ &= -2k < 0 \end{align*} and so if there is a vertex of degree $q-1$ or less in $A$ then $M$ is not minimal.

Thus we see that both $M$ and $\xs$ are minimized when every vertex in $A$ has degree $q$ or degree $q+1$. Thus, if \ref{2path} does not hold in that case, it cannot hold in any case. To obtain values for $\xs$ and $\ys$, we solve this system of equations: \begin{align*} \ws + \xs + \ys + 1 &= n \\ \ws(q+2) + \xs(q+1) + \ys q + \delta &=  2(E_0+1) \end{align*} which makes $\xs = q^2 + 2 - \delta - 2\ws$ and $\ys = q -3 +\ws +\delta$.

When we plug those values into \ref{2path} and group terms in terms of $\delta$ we obtain this expression: $$-\frac 1 2 \delta^2 +(q+\frac 3 2)\delta + \frac 3 2 \ws - \frac 3 2 q - \frac 1 2 q^2 - 2 \geq 0$$ Since, viewed as a function of $\delta$, that is a downward-opening quadratic, we know that the inequality will only be true between the zeros of that function. Applying the quadratic formula to the expression yields this equivalent expression: $$\frac 3 2 + q + \frac 1 2 \sqrt{-7 + 12\ws} \geq \delta \geq \frac 3 2 + q - \frac 1 2 \sqrt{-7 + 12\ws}$$ and since we know that $\delta \geq \ws$, we know this must be true: \begin{align*} \frac 3 2 + q + \frac 1 2 \sqrt{-7 + 12\ws} &\geq \ws \\ \frac 3 2 + q + \frac 1 2 \sqrt{-7 + 12\ws} -\ws &\geq 0 \end{align*}

Solving that inequality for $\ws$ yields a requirement that $$\ws \geq 3+q - \sqrt{5+3q}$$ but that (given our constraint that $q$ be at least 6) implies $\ws > \frac q 2 + 1$. That means that $\delta > \frac q 2 + 1$, which contradicts our initial assumption. Therefore $\delta(G) > \frac q 2 + 1$. \end{proof}

\begin{lemma}
If $G$ is a $C_4$-free graph on $n$ vertices with $E_0+1$ edges (with no restrictions on $q$), then any two vertices of degree $q+2$ in $G$ must share exactly one neighbor. \end{lemma}

\begin{proof} By way of contradiction, suppose that there exist two vertices $u,v\in G$ both of degree $q+2$ that share no neighbors.
We will expand on a technique used by F\"uredi in \cite{F1} to consider 2-paths without endpoints in either $\Gamma(u)$ or $\Gamma(v)$. Denote this quantity by $P$. Let $d=q+2$ be the degree of $u$ and $v$.

We know that ${{n-d(x)} \choose 2}$ is an upper bound for the number of 2-paths without endpoints in the neighborhood of a chosen vertex $x$ with degree $d(x)$, since every two vertices not in the neighborhood of $x$ are endpoints of at most one 2-path. In our case, we have an upper bound on $P$ of ${{n-2d} \choose 2}$ as we are removing two disjoint neighborhoods of degree $d$.

Now, we know that $\sum_{w\neq u,v}{{d(w)} \choose 2}$ is precisely the number of 2-paths with the central vertex not equal to $u,v$, but we must subtract two from each degree to account for the possibility that a given $w$ shares a neighbor with both $u$ and $v$. Thus we find that $\sum_{w\neq u,v}{{d(w)-2} \choose 2}$ is the lower bound for $P$, as it assumes all vertices $w$ share a neighbor with both $u$ and $v$.
We use Jensen's inequality to get the following result: $$\sum_{w\neq u,v}{{d(w)-2} \choose 2}\geq (n-2){{(2(E_0+1)-2(n-2)-2d)/(n-2)} \choose 2}.$$
Then we have an inequality that must hold for the graph $G$ to exist: $${{n-2d}\choose 2}-1\geq P \geq (n-2){{(2(E_0+1)-2(n-2)-2d)/(n-2)} \choose 2}.$$

After simplifying and solving for $q$, we get the following inequality: $$\frac{-q^4-6q^3+17q^2+34q-48}{q^2+q-2}\geq 0.$$ However, this inequality cannot hold for any $q$ in the range we are concerned with. This contradiction shows that $u$ and $v$ must share at least one neighbor, and we know they cannot share more than one because this would create a $C_4$. This result implies that every pair of vertices of degree $q+2$ must have exactly one neighbor in common. \end{proof}

\begin{lemma}
If $G$ is a $C_4$-free graph on $n$ vertices with $E_0+1$ edges (where $q$ can be even or odd), any two vertices of degree $q+2$ must share a neighbor of degree $d < \frac q 2$. 
\end{lemma}

\begin{proof}
We know from the previous lemma that two vertices of degree $q+2$ must have a neighbor in common. Consider two such vertices $x$ and $y$, and let their unique common neighbor be $u$. By adapting the argument of the previous lemma and applying it to $B = \Gamma(x) \cup \Gamma(y)$, i.e. looking at $P'$, the number of 2-paths with no endpoints in $B$, we obtain the following inequality:
$${{n-2(q+2)+1} \choose 2} \geq \sum_{v \not\in \Gamma(u)} {{d(u)-2} \choose 2} + \sum_{v \in \Gamma(u),v\neq x,y} {{d(u)-1} \choose 2}$$
We see the left hand side is an upper bound on $P'$; we add 1 back in to the number of vertices used in 2-paths because of the intersection between $\Gamma(x)$ and $\Gamma(y)$. The right hand side is separated into two sums.  The first summation sums over all vertices $v$ not in the neighborhood of vertex $u$, and relies on the fact that $v$ can connect to at most 1 one other vertex in the neighborhood of each of the $q+2$ vertices before creating a $C_4$.  The second summation is over all vertices $w \in \Gamma(u)$ and comes from the fact that $w$ can connect to at most 1 other vertex in the union of the neighborhoods of the two $q+2$ vertices before creating a $C_4$.

We use Jensen's inequality on the right hand side to obtain the following expression:
$$\sum_{v \not\in \Gamma(u)} {{d(v)-2} \choose 2} + \sum_{w \in \Gamma(u)} {{d(w)-1} \choose 2} \geq (n-2){{\frac{2e-2(q+2)-2(n-2)+d(u)-2}{n-2}} \choose 2}$$ (We subtract 2 because we should not count $x$ and $y$ for $u$'s total.) 

Clearly, if this inequality fails for a given value of $d(u)$, it must fail for any larger value, since the right side increases with $d(u)$ and the left side is static. Thus, we plug in $\frac q 2$, which yields the following expression: $$- \frac{6q^3-25q^2-28q+96}{8q^2+8-16} \geq 0$$ 

That inequality fails for all relevant $q$, thus the statement is proven. \end{proof}

\begin{corollary} \label{noqp2}
If $q$ is even, any $C_4$-free graph with $E_0+1$ edges and $n$ vertices has exactly one vertex of degree $q+2$.
\end{corollary}

\begin{proof} This follows from the three previous lemmas. \end{proof}


Having reduced the hypothetical counterexamples to a single case, we proceed with the proof of the theorem.

\begin{theorem}
For $q$ even, $ex(q^2+q,C_4) \leq \frac 1 2 q(q+1)^2-q$.
\end{theorem}

\begin{proof} We know from Corollary \ref{noqp2} that we must consider only the case when $\ws = 1$. It is clear that any graph $G$ with $n$ vertices, $E_0+1$ edges, and $\ws = 1$ can only be created by taking a graph $G'$ with $E_0$ edges and $\Delta(G) = q+1$ and connecting a vertex of degree $q+1$ to a vertex of strictly lower degree. We know from Lemma \ref{degseqs} that $G'$ can have one of two possible degree sequences up to a parameter $z$. When we examine all the ways to make the necessary connection, we get these four possible degree sequences for $G$:
\begin{center}
    \begin{tabular}{ | l | l | l | l | l | l |}
    \hline
    & $|X_{q+2}|$ & $|X_{q+1}|$ & $|X_q|$ &$|X_{q-1}|$ &$|X_{q-2}|$ \\ \hline
    A & 1 & $q^2-q+z$ & $2q-2z-1$ & $z$ & 0\\ \hline 
    B & 1 & $q^2-q+z$ & $2q-2z$ & $z-2$ & 1\\ \hline 
    C & 1 & $q^2-q+z+1$ & $2q-2z-2$ & $z-1$ & 1\\ \hline
    D & 1 & $q^2-q+z-1$ & $2q-2z+1$ & $z-1$ & 0\\ \hline 
    \end{tabular}
\end{center}

Since we have a specific degree sequence, we can use Lemma \ref{total2plemma} concerning the total number of 2-paths in $G$ to generate the following inequality: $$qe - \xs + \frac 1 2 \ws \geq \ws {{q+2} \choose 2}$$ $$+ \ \xs{{q+1} \choose 2} + \ys {q \choose 2} + |X_{q-1}|{{q-1} \choose 2} + |X_{q-2}|{{q-2} \choose 2}$$

When we solve that inequality for $z$, we get the following results: 

\begin{center}
{

    \begin{tabular}{| l | l |}
    \hline
 \bigstrut   A & $z \leq - \frac 1 4$ \\ \hline
 \bigstrut   B & $z \leq -\frac 3 4 $ \\ \hline
 \bigstrut   C & $z \leq -\frac 7 4 $ \\ \hline
\bigstrut    D & $z \leq \frac 3 4 $ \\ \hline
    \end{tabular}}
\end{center}

Obviously, these values of $z$ lead to impossible degree sequences, thus no such $G$ is possible. \end{proof}

\begin{acknowledgements}
This research was conducted at an NSF Research Experience for Undergraduates (grant number DMS-0755450) at the University of Wyoming in the summer of 2011. We would like to thank Colin Garnett, Bryan Shader, and everyone else affiliated with the program for their assistance.
\end{acknowledgements}

\end{document}